\documentclass[11pt]{amsart}
\usepackage{geometry}                
\usepackage{graphicx}
\usepackage{amsmath,amsthm,amssymb,enumerate, multicol, fullpage, setspace, dsfont, mathrsfs}
\usepackage{epstopdf}
\theoremstyle{plain}
\newtheorem{theorem}{Theorem}
\newtheorem{corollary}{Corollary}
\newtheorem{lemma}{Lemma}

\theoremstyle{definition}

\title{Triangulable Leibniz Algebras}
\author{Tiffany Burch and Ernie Stitzinger}

\begin{document}
\maketitle

ABSTRACT: A converse to Lie's theorem for Leibniz algebras is found and generalized. The result is used to find cases in which the generalized property, called triangulable, is 2-recognizeable; that is, if all 2-generated subalgebras are triangulable, then the algebra is also. Triangulability joins solvability, supersolvability, strong solvability, and nilpotentcy as a 2-recognizeable property for classes of Leibniz algebras.

\vspace{5mm}

Simultaneous triangulation of a set of linear transformations related to Leibniz algebras will be considered in this work. A necessary condition is that the minimum polynomials of the linear transformations are the products of linear factors. If the matrices commute, then this condition is sufficient. Lie's theorem is a generalization of this result, a result that is field dependent even in the algebraically closed case. Lie's theorem extends to Leibniz algebras. It can also be useful to have a condition on the set which guarantees simultaneous triangulation when extending to the algebraic closure, $K$, of the field of scalars, $F$. We will call this property, borrowing from Bowman, Towers and Varea \cite{Bowman}, triangulable. Thus a Leibniz algebra $A$ over a field $F$ is triangulable on module $M$ if the induced representation of $K \bigotimes A$ on $K \bigotimes M$ admits a basis such that the representing matrices are upper triangular.   We find a condition on the Leibniz algebra that is equivalent to being triangulable. The condition is independent of the field and includes Lie's theorem. The condition is on $A$ although triangulable refers to the action in the algebra that is over $K$. As an application, we consider to what extent an algebra is triangulable if all 2-generated subalgebras have the property. Such a property is called 2-recognizable, a concept investigated in groups, Lie algebras and Leibniz algebras for classes such as solvable, supersolvable, strongly solvable and, trivially, nilpotent and abelian. For an introduction to Leibniz algebras, see \cite{Barnes}, \cite{Demir} and \cite{Loday}.

\vspace{5mm}

Let $M$ be a module for the Leibniz algebra $A$. Let left and right multiplication by $x \in A$ be denoted by $T_x$ and $S_x$, respectively. If $T_x$ is nilpotent, then $S_x$ is also nilpotent and an Engle theorem holds; that is, $A$ acts nilpotently on $M$ \cite{Bosko}. Define $A$ to be nil on $M$ if this property holds. If $I$ is an ideal of a Leibniz algebra $A$, then $I$ is nil on $A$ if and only if $I$ is nilpotent. We will show that $A$ is triangulable on itself if and only if $A^2$ is nilpotent. When $A^2$ is nilpotent, then $A$ is called strongly solvable, a term introduced in \cite{Bowman}.

\vspace{5 mm}

\begin{lemma}\label{Center}
Let $B$ be an irreducible submodule of the $A$-module $M$ such that dim$(B) = 1$. Then $A^2 \subseteq C_A(B)=\{a \in A | aB=0=Ba\}$.
\end{lemma} 

\begin{proof}
Let $x, y \in A$, $b \in B$ with $xb=\alpha_x b$ where $\alpha_x$ is a scalar. Then $(xy)b=x(yb)-y(xb)=x(\alpha_y b)-y(\alpha_x b)= \alpha_x \alpha_y b -\alpha_y \alpha_x b=0$. Since $B$ is irreducible, it follows that $b(xy)=0$ (\cite{Barnes}). Hence $A^2 B=B A^2=0$ and the result holds.
\end{proof}

\vspace{5mm}

\begin{theorem}\label{ThmBasis}
Let $A$ be a Leibniz algebra and $M$ be an $A$-module. There is a basis for $M$ such that the representing matrices  are in upper triangular form if and only if the linear transformations in the representation have minimum polynomials with all linear factors and $A^2$ is nil on $M$.
\end{theorem}

\vspace{5mm}

We then have

\vspace{5mm}

\begin{corollary}\label{CorBasis}
Let $A$ be a Leibniz algebra over an algebraically closed field and $M$ be an $A$-module. There is a basis for $M$ such that the representing matrices are in upper triangular form if and only if $A^2$ is nil on $M$.
\end{corollary}

\begin{proof}[Proof of Theorem]
The condition on the minimum polynomials is equivalent to being able to triangulate all the linear transformations. Suppose also that $A^2$ is nil on $M$. Then $T_x$, and hence $S_x$,  is nilpotent for all $x \in A^2$ (\cite{Bosko}) and $A^2$ acts nilpotently on $M$ by Engel's theorem. Let $B$ be an irreducible submodule of $M$. It is enough to show that dim$(B) =1$ for then the result will follow by induction. Since the submodule $BA^2+A^2B$ is contained in $B$, $BA^2+A^2B$ is 0 or $B$. If the latter holds, then $B=BA^2+A^2B=(BA^2+A^2B)A^2+A^2(BA^2+A^2B)=\dots$, which is never 0, a contradiction. Hence $BA^2+A^2B=0$ and $BA^2=0=A^2B$. Therefore, $0=(xy)b=x(yb)-y(xb)$ for all $x, y \in A$, $b \in B$. Hence $x(yb)=y(xb)$ and $T_xT_y=T_y T_x$. Since $B$ is irreducible, either $S_x=0$ for all $x \in A$ or $S_x=-T_x$ for all $x \in A$ \cite{Barnes}. Then either $T_xS_y=-T_xT_y=-T_yT_x=S_yT_x$ or $T_xS_y=0=S_yT_x$ on $B$ and $S_xS_y=T_xT_y=T_yT_x=S_yS_x$ or $S_xS_y=0=S_yS_x$ on $B$. Hence all $T_x$, $S_y$ commute on $B$. Thus these transformations are simultaneously triangularizeable on $B$. Hence dim$(B)=1$ and the result holds in this direction.

\vspace{5mm}

Suppose that there is a basis for $M$ such that all the transformations are triangularizeable on this basis. Then there is a one dimensional submodule $B$ such that $A^2B=0=BA^2$ by the lemma. $A^2$ is nil on $M/B$ by induction and hence on $M$. Since each transformation is triangularizeable, the minimum polynomials have linear factors.
\end{proof}

\vspace{5mm}

We state these results when $A$ is acting on itself. $A$ is called strongly solvable if $A^2$ is nilpotent (\cite{Bowman}, \cite{Burch}). $A$ is supersolvable if there is an increasing chain of ideal of $A$, each of codimension 1 in the next. 

\vspace{5mm}

\begin{theorem}\label{Supersolvable}
Let $A$ be a Leibniz algebra over an algebraically closed field. Then $A$ is supersolvable if and only if $A$ is strongly solvable.
\end{theorem}

\begin{proof}
$A$ is supersolvable if and only if $A$ triangulable on itself. By Corollary \ref{CorBasis}, this is equivalent to $A^2$ being nil on $A$ which in turn is the same as $A^2$ being nilpotent. The last statement is the definition of strongly solvable.  
\end{proof}

\vspace{5mm}

The general case looks like the following.

\vspace{5mm}

\begin{theorem}\label{Triangulable}
$A$ is triangulable if and only if $A^2$ is nil on $A$.
\end{theorem}

\begin{proof}
Let $A^*=K \bigotimes A$ where $K$ is the algebraic closure of the base field for $A$. If $A^2$ is nil on $A$, then $A^2$ is nilpotent as is $(A^2)^*=(A^*)^2$. Hence $A^*$ is supersolvable by the last result and $A$ is triangulable by definition. The converse holds by reversing the steps.
\end{proof}

\vspace{5mm}

A property is $n$-recognizeable in a class of algebras if whenever all $n$ generated subalgebras have the property, then the algebra has the property. The concept when $n=2$ has been considered in groups, Lie algebras and Leibniz algebras for classes that are solvable, supersolvable, nilpotent and strongly solvable (see \cite{Brandl}, \cite{Burch}, \cite{Moneyhun}). A version of this has also been considered for triangulable in the Lie case in \cite{Bowman}. We considered it as an application of our result.

\vspace{5mm}

\begin{theorem}\label{TriRec}
In the class of solvable Leibniz algebras, triangulable is 2 recognizable.
\end{theorem}

\begin{proof}
Let $N$ be a 2-generated subalgebra of $A$. Since $N$ is triangulable, $N^2$ is nil on $N$ and, thus, $N$ is strongly solvable. Therefore $A$ is strongly solvable by \cite{Burch} and $A^2$ is nil on $A$. Then, using Theorem \ref{Triangulable}, $A$ is triangulable.
\end{proof}

\vspace{5mm}

The result can be extended from solvable to all algebras in characteristic 0, again using a result in \cite{Burch}.

\vspace{5mm}
\begin{corollary}\label{TriChar}
For Leibniz algebras, triangulable is 2-recognizeable at characteristic 0.
\end{corollary}

\begin{proof}
Let $A$ be a minimal counterexample. Simple Leibniz algebras are Lie algebras. Simple Lie algebras are 2-generated, hence $A$ is triangulable. If $I$ is a proper ideal in $A$, then all 2-generated subalgebras of $I$ and $A/I$ are strongly solvable by Theorem \ref{Triangulable} and $I$ and $A/I$ are strongly solvable by Theorem 3 of \cite{Burch} and $A$ is solvable. The result holds by Theorem \ref{TriRec}.
\end{proof}

\end{document}